\documentclass[12pt]{article}

\usepackage[utf8]{inputenc}
\usepackage[T1]{fontenc}
\usepackage{lmodern}
\usepackage[english]{babel}
\usepackage{amsmath,mathtools,amsthm,amssymb,amsfonts}
\usepackage{color}
\usepackage{comment}

\usepackage{hyperref}
\usepackage{mathrsfs}
\usepackage{graphicx}
\usepackage{float}
\usepackage{enumitem}
\usepackage{tikz} 
\usepackage{esint}
\usepackage{pgfplots}
\usepackage{esvect}
\usepackage{pdfpages}
\usepackage{nicefrac}
\usepackage{leftidx, tensor}
\usepackage{cleveref}
\usepackage{accents}
\usetikzlibrary{shapes,arrows}
\usetikzlibrary{calc}
\usetikzlibrary{shapes.geometric}
\usetikzlibrary{shapes.arrows}
\usetikzlibrary{arrows.meta}
\usetikzlibrary{decorations.markings}
\usetikzlibrary{fit}
\usetikzlibrary{patterns}
\usetikzlibrary{hobby}
\usepgfplotslibrary{patchplots}
\pgfplotsset{compat=1.11}

\renewcommand{\thefootnote}{\fnsymbol{footnote}}
\newcommand{\definedas}{\mathrel{\raise.095ex\hbox{\rm :}\mkern-5.2mu=}}

\newcommand{\R}{\mathbb{R}}

\newcommand{\Sbb}{\mathbb{S}}

\renewcommand{\d}{\,\mathrm{d}}

\newcommand{\ul}[1]{\underline{#1}}

\newcommand{\btr}[1]{\left\vert#1\right\vert}
\newcommand{\newbtr}[1]{\vert#1\vert}

\newcommand{\spann}[1]{\left\langle#1\right\rangle}

\newcommand{\Ric}{\mathrm{Ric}}

\newcommand{\Rm}{\mathrm{Rm}}

\newcommand{\two}{\operatorname{II}}
\newcommand{\tr}{\text{tr}}

\theoremstyle{plain}
\newtheorem{thm}{Theorem}[section]
\newtheorem{prop}[thm]{Proposition}

\newtheorem{lem}[thm]{Lemma}

\theoremstyle{definition}

\newtheorem{bem}[thm]{Remark}
\newtheorem{kor}[thm]{Corollary}

\begin{document}
		\begin{center}
			{\LARGE {An area growth argument for null mean curvature flow along the standard de\,Sitter lightcone}\par}
		\end{center}
		\vspace{0.5cm}
		\begin{center}
			{\large Markus Wolff\footnote[2]{markuswo@kth.se}}\\
			\vspace{0.4cm}
			{\large Department of Mathematics}\\
			{\large KTH Royal Institute of Technology}
		\end{center}
		\vspace{0.4cm}
		\begin{abstract}
			We consider null mean curvature flow along the standard lightcone in the de\,Sitter spacetime. This flow was first studied by Roesch--Scheuer along null hypersurfaces for the detection of MOTS, and independently by the author in the specific case of the standard Minkowski lightcone. Similar to the Minkowski case, null mean curvature flow along the de\,Sitter lightcone can be related to $2d$-Ricci flow for surfaces of genus $0$ by an appropriate rescaling. Building on this rescaling procedure, we analyse singularity formation, asymptotic behavior and ancient solutions to the flow.
		\end{abstract}
		\renewcommand{\thefootnote}{\arabic{footnote}}
		\setcounter{footnote}{0}
	\section{Introduction}
		For a given null hypersurface $\mathcal{N}$ in an ambient spacetime $(M,g)$ and a choice of null generator $\ul{L}$, we say a family of spacelike cross sections $(\Sigma_t)$ smoothly evolves along $\mathcal{N}$ under what we will call \emph{null mean curvature flow} here if their exists a smooth embedding $x\colon (0,T)\times \Sigma_0\to \mathcal{N}$ such that $\Sigma_t=x(t,\Sigma_0)$ and 
		\[
			\frac{\d }{\d t}x=\frac{1}{2}g(\vec{\mathcal{H}},L)\ul{L},
		\]
		where $\vec{\mathcal{H}}$ denotes the codimension-$2$ mean curvature vector of the spacelike cross sections in $(M,g)$ and $L$ is the unique null vector field normal to the spacelike cross sections satisfying $g(\ul{L},L)=2$. Note that null mean curvature flow is the projection of codimension-$2$ mean curvature flow onto $\mathcal{N}$ in direction of the null generator $\ul{L}$, and that the flow is at least locally always equivalent to a scalar parabolic equation, cf. \cite{roeschscheuer}.
		
		This extrinsic flow was first studied by Roesch--Scheuer \cite{roeschscheuer} to detect \emph{marginally outer trapped surfaces} (MOTS). Such surfaces are models for apparent horizons in null hypersurfaces and in particular arise as the intersection between the null hypersurface and a (Killing) horizon in the ambient spacetime. Under some mild assumptions on the choice of null generator $\ul{L}$ and assuming that appropriate barriers exists, Roesch--Scheuer show that the flow exists for all times and converges smoothly to a MOTS, cf. \cite[Theorem 1.1]{roeschscheuer}. In particular, their assumptions are satisfied for the lightcone in the Schwarzschild spacetime (of positive mass). In a recent paper \cite{wolff1}, the author further considered the flow in the round Minkowski lightcone. No closed MOTS exist in the Minwkowski spacetime and one thus expects the flow to develop singularities in finite time. This is indeed the case, as one can show the equivalence between the extrinsic null mean curvature flow in the Minkowski lightcone and $2d$-Ricci flow in the conformal class of the round metric using the Gauss equation. Thus, a classical result first proven by Hamilton \cite{hamilton1} leads to a complete understanding of the singularity formations in the case of the round Minkowski lightcone, cf. \cite[Corollary 11]{wolff1}.
		
		In this paper, we consider null mean curvature flow in the standard lightcone of the de\,Sitter spacetime. The de\,Sitter spacetime is a vacuum solution of the Einstein Equation with positive cosmological constant $\Lambda=1$ and serves as s cosmological model for Einstein's Theory of General Relativity. Here, we consider the de\,Sitter spacetime as a static, spherically symmetric spacetime $(M,g)$ of the form
		\begin{align*}
			M&=\R\times(0,1)\times\Sbb^2,\\
			g&=-(1-r^2)\d t^2+\frac{1}{1-r^2}\d r^2+r^2\d\Omega^2.
		\end{align*}
		As we want to consider the lightcone for all positive radii in $(0,\infty)$, we will extend the spacetime beyond the cosmological Killing horizon $\{r=1\}$ into a generalized Kruskal--Szekeres spacetime covered by double null coordinates, cf. \cite{cedwolff}. On the other hand, one can think of the de\,Sitter spacetime as the \emph{one-sheeted hyperboloid} or \emph{pseudosphere} $\mathbb{H}_1^{1,3}$ in the $1+4$\,-dimensional Minkowski spacetime $\R^{1,4}$. It is easy to see that the static coordinates as above do not cover all of $\mathbb{H}_1^{1,3}$, and that similar to the generalized Kruskal--Szekeres extension $\mathbb{H}_1^{1,3}$ contains two copies of $(M,g)$. Indeed, one can see that the generalized Kruskal--Szekeres extension recovers a double null coordinate system that covers all of $\mathbb{H}_1^{1,3}$ (expect for a set of measure zero). See Section \ref{sec_discussion} below for more details.
		
		As the de\,Sitter spacetime admits a similar metric structure as the Minkowski spacetime in the above coordinates, several properties of the standard Minkowski lightcone carry over to the standard de\,Sitter lightcone. More generally, we recall the well-known fact that the principal null hypersurfaces in a class $\mathcal{S}$ of static, spherically symmetric spacetimes, cf. Cederbaum--Galloway \cite{cedgal}, yield a wide array of examples of shear-free lightcones with conformally round cross sections, see Section \ref{sec_lightcone}. Spacetimes of class $\mathcal{S}$ and generalizations have been widely studied, see e.g. \cite{birmi, brillhayward,cedgal, cedwolff,chenwang, chrugalpot,kottler, schindagui, wangwangzang}.
		In particular, the Gauss Equation for a spacelike cross sections $\Sigma$ of the de\,Sitter lightcone yields
		\begin{align}\label{eq_intro1}
			\mathcal{H}^2=2\operatorname{R}-4,
		\end{align}
		where $\mathcal{H}^2:=g(\vec{\mathcal{H}},\vec{\mathcal{H}})$ is the \emph{spacetime mean curvature}, and $\operatorname{R}$ denotes the scalar curvature of $\Sigma$, respectively, see Corollary \ref{kor_gaussdesitter}. Similar to the case of the Minkowski lightcone, the Gauss Equation \eqref{eq_intro1} now infers that, while not directly equivalent, null mean curvature flow can be related to $2d$-Ricci flow in the conformal class of the round sphere upon a suitable rescaling. This again yields a complete understanding of the asymptotic behavior and the formation of singularities for null mean curvature flow along the de\,Sitter lightcone, see Theorem \ref{thm_main1}. Unlike the case of the Minkowski lightcone, we show that not all solutions develop singularities in finite time, but may exists for all times and either converge smoothly to a round MOTS, or expand towards infinity without becoming round. Our analysis of the flow relies heavily on an explicit formula of the area $\btr{\Sigma_t}$, which gives us precise control over the timescale of the rescaled equation.
		
		Note that by \eqref{eq_intro1}, any MOTS in the de\,Sitter lightcone is given as a conformally round surface with constant scalar curvature $\operatorname{R}=2$, and thus area $4\pi$, and which arises as the intersection between the lightcone and a cosmological Killing horizon. In particular, unlike the cases considered by Roesch--Scheuer \cite{roeschscheuer}, any MOTS in the de\,Sitter lightcone turns out to be unstable by Theorem \ref{thm_main1} in the following sense: Any slight variation $\Sigma_0$ of a MOTS such that $\btr{\Sigma_0}\not=4\pi$ will completely detach itself from the MOTS and either shrink towards the tip of the cone in finite time or expand towards infinity.
		
		Additionally, we consider ancient solutions of null mean curvature flow on the de\,Sitter lightcone, where we call a solution ancient if it exists for all negative times as usual. As before, we utilize the rescaled equation to relate any ancient solution of the flow to a conformally round ancient solution of $2d$-Ricci flow. Apart from the flat, eternal solution, i.e., a solution that exists for all times, all compact ancient solutions of $2d$-Ricci flow are conformally round and have been fully classified as either shrinking spheres or a King--Rosenau solution by Daskalopoulos--Hamilton--Sesum \cite{daskahamilsesum}. Arguing by rescaling, their analysis also leads to a full classification of ancient solution to null mean curvature flow, see Theorem \ref{thm_main2}. More precisely, there are $6$ ancient solutions of null mean curvature flow, of which $4$ are eternal solutions, and such that each of them corresponds either to a family of shrinking spheres or a King--Rosenau solution upon rescaling.
		
		We briefly comment on the higher dimensional case and the case of the round lightcone in the Anti de\,Sitter spacetime in Section \ref{sec_discussion}.\newpage
		
		This paper is structured as follows:\newline
		In Section \ref{sec_prelim} we fix our notation and recall some preliminaries. In Section \ref{sec_lightcone} we discuss shear-free null hypersurfaces, which we will call (standard) lightcones here, in a class $\mathcal{S}$ of static spacetimes and their generalized Kruskal--Szekeres extension, in particular in the case of the de\,Sitter spacetime. We discuss null mean curvature flow along the de\,Sitter lightcone in Section \ref{sec_nullMCF}, and prove our main results Theorem \ref{thm_main1} and Theorem \ref{thm_main2}. We close with some comments on the relation between the generalized Kruskal--Szekeres extension and the pseudosphere $\mathbb{H}_1^{1,3}$, the higherdimensional case, and the case of the Anti de\,Sitter lightcone in Section \ref{sec_discussion}.
		
		\subsection*{Acknowledgements.}
		I would like to express my sincere gratitude towards Carla Cederbaum and Gerhard Huisken for their continuing guidance and helpful discussions.
		Additional thanks to Giorgos Vretinaris for his enlightening figure.
		\setcounter{section}{1}
		
\section{Preliminaries}\label{sec_prelim}
	For the purpose of this paper, $(\Sigma,\gamma)$ will always denote a spacelike, closed, orientable surface of codimension-$2$ in an ambient spacetime $(M,g)$. Recall that for such a spacelike, codimension-$2$ surface, the \emph{vector-valued second fundamental form} of $\Sigma$ in $(M,g)$ is defined as
	\[
	\vec{\two}(V,W)=\left(\overline{\nabla}_VW\right)^\perp
	\]
	for all tangent vector fields $V, W\in\Gamma(T\Sigma)$, where $\overline{\nabla}$ denotes the Levi-Civita connection of $(M,g)$. Further, the \emph{codimension-$2$ mean curvature vector} $\vec{\mathcal{H}}$ of $\Sigma$ is given by the trace of $\vec{\two}$ with respect to $\gamma$, i.e, $\vec{\mathcal{H}}=\tr_\gamma\vec{\two}$. Additionally, we define the \emph{spacetime mean curvature} $\mathcal{H}^2$ as the Lorentzian length of $\vec{\mathcal{H}}$, i.e., 
	\[
		\mathcal{H}^2=g(\vec{\mathcal{H}},\vec{\mathcal{H}}).
	\]
	If $\vec{\mathcal{H}}$ is spacelike, this agrees with the notion of spacetime mean curvature by Ceder\-baum--Sakovich \cite{cederbaumsakovich} upon taking a square root. However, as $\vec{\mathcal{H}}$ will in general have no  fixed causal character, $\mathcal{H}^2$ can be at least locally negative. Indeed, \emph{trapped surfaces}, where $\vec{\mathcal{H}}$ is timelike everywhere along $\Sigma$, and hence $\mathcal{H}^2<0$ globally, will naturally occur in the de\,Sitter lightcone, cf. Remark \ref{bem_gaussdesitter} (ii). Additionally, we call $\Sigma$ a surface of constant spacetime mean curvature (STCMC surface) if $\mathcal{H}^2$ is constant along $\Sigma$, cf. \cite{cederbaumsakovich}.
	
	For a choice of null frame $\{\ul{L},L\}$ with $g(\ul{L},L)=2$, $\vec{\operatorname{II}}$ and $\vec{\mathcal{H}}$ admit the decomposition 
	\begin{align}
	\begin{split}
	\vec{\operatorname{II}}&=-\frac{1}{2}\chi\underline{L}-\frac{1}{2}\underline{\chi}L,\\
	\vec{\mathcal{H}}&=-\frac{1}{2}\theta\underline{L}-\frac{1}{2}\underline{\theta}L,
	\end{split}
	\end{align}
	where the null second fundamental forms $\underline{\chi}$ and $\chi$ with respect to $\underline{L}$ and $L$, respectively, are defined as
	\begin{align*}
	\underline{\chi}(V,W):=-g\left( \overline{\nabla}_VW,\underline{L} \right)&\,&
	{\chi}(V,W)&:=-g\left( \overline{\nabla}_VW,L \right),
	\end{align*}
	for tangent vector fields $V,W\in \Gamma(T\Sigma)$, and the null expansions $\underline{\theta}$ and $\theta$ with respect to $\underline{L}$ and $L$, respectively, as
	\begin{align*}
	\underline{\theta}:= \operatorname{tr}_\gamma\underline{\chi},&\,&
	{\theta}:= \operatorname{tr}_\gamma{\chi}.
	\end{align*}
	In particular
	\begin{align}\label{eq_secondffnulldecomp}
	\mathcal{H}^2&=\underline{\theta}\theta.
	\end{align}
	Further, we define the connection-$1$ form $\zeta$ as 
	\[
	\zeta(V):=\frac{1}{2}g\left(\overline{\nabla}_V\underline{L},L \right).
	\]
	Throughout this paper, we will usually assume that $\Sigma$ arises as a spacelike cross section of a null hypersurface $\mathcal{N}$ and such that $\ul{L}$ corresponds to a choice of null generator of $\mathcal{N}$. In this case, we additionally assume that any integral curve of $\ul{L}$ in $\mathcal{N}$ intersects $\Sigma$ exactly once. In the next section, this will allow us to explicitly derive the above objects in the lightcones under consideration. However,  one should point out that the explicit formulas that we derive in this special case build on a framework that is more generally available for null hypersurfaces. For more details and further references, we refer the interested reader to \cite{marssoria, roeschscheuer, sauter}.
	
	Note that we use the following conventions for the Riemann curvature tensor $\operatorname{Rm}$, Ricci curvature tensor $\operatorname{Ric}$ and scalar curvature $\operatorname{R}$, respectively:
	\begin{align*}
	\Rm(X,Y,W,Z)&=\spann{\nabla_X\nabla_YZ-\nabla_Y\nabla_XZ-\nabla_{[X,Y]}Z,W},\\
	\Ric(X,Y)&=\tr\Rm(X,\cdot,Y,\cdot),\\
	R&=\tr \Ric.
	\end{align*}

\section{The de\,Sitter lightcone}\label{sec_lightcone}
	We now introduce the (standard) de\,Sitter lightcone on which we want to perform our analysis of null mean curvature flow. In particular, we define it as a \emph{shear-free null hypersurface} that arises in a suitable spacetime extension in double null coordinates. In fact, we consider such lightcones more generally in a class $\mathcal{S}$ of spherically symmetric spacetimes, cf. \cite{cedgal}. 
	
	Now, let $h\colon (0,\infty)\to\mathbb{R}$ be a smooth function with finitely many, positive zeros \linebreak${r_0:=0<r_1<\dotsc<r_N<r_{N+1}:=\infty}$. Then, we say that $(M_i,g)$ with
	\begin{align*}
		M_i&:=\R\times(r_{i-1},r_i)\times\Sbb^2,\\
		g&:=-h(r)\d t^2+\frac{1}{h(r)}\d r^2+r^2\d\Omega^2
	\end{align*}
	is a \emph{spacetime of class $\mathcal{S}$} for $i\in\{1,\dotsc,N+1\}$, where $\d\Omega^2$ denotes the standard round metric on $\Sbb^2$. Spacetimes in Class $\mathcal{S}$ and generalizations have been extensively studied, see e.g. \cite{birmi, brillhayward,cedgal, cedwolff,chenwang, chrugalpot,kottler, schindagui, wangwangzang}. Assuming that the null energy is satisfied, Wang--Wang--Zhang \cite{wangwangzang} prove a spacetime Alexandrov Theorem in Class $\mathcal{S}$ characterizing shear-free null hypersurfaces, and Chen-Wang \cite{chenwang} provide a characterization of STCMC surfaces on these shear-free null hypersurfaces. Their analysis is performed in (global) $(t,r)$-coordinates on a spacetime $(M_i,g)$. These coordinates fail for any zero $r_i$ and the sets $\{r=r_i\}$ correspond to Killing horizons (in a suitable spacetime extension). If $h'(r_i)\not=0$, which corresponds to the fact that the Killing horizon $\{r=r_i\}$ is \emph{non-degenerate}, cf. \cite[Equation (12.5.16)]{wald}, there exists a generalized Kruskal--Szekeres extension joining $M_{i}$, $M_{i+1}$ as first constructed in the general case by Brill--Hayward \cite{brillhayward}, and independently by Schindler--Aguirre \cite{schindagui}, and Carla Cederbaum and the author \cite{cedwolff}. In the specific case of the de\,Sitter spacetime, Kruskal coordinates were also considered in \cite{gibbhawk}. We briefly introduce the construction in \cite{cedwolff} as we want to study the shear-free null hypersurfaces considered in \cite{wangwangzang,chenwang} in the resulting double null coordinates both for notational convenience as well as for the fact that they naturally extend past the Killing horizon in these coordinates.
	
	In \cite{cedwolff}, the contruction of a generalized Kruskal--Szekeres extension is reduced to finding a strictly increasing solution $f_i$ of 
	\begin{align}\label{eq_ODEkruskal}
		\frac{f_i(r)}{f'_i(r)}=K_ih(r)
	\end{align}
	on $(r_{i-1},r_{i+1})$ with $K_i=\frac{1}{h'(r_i)}\not=0$. By \cite[Proposition 3.2]{cedwolff} a solution exists and is unique (up to a multiplicative constant) if and only if $h'(r_i)\not=0$. Then $M_i$, $M_{i+1}$ are contained in the spacetime extension $(\mathbb{P}_h^i\times\Sbb^2,\widetilde{g})$, where
	\begin{align*}
		\mathbb{P}_h^i&=\{(u,v)\in\R^2\colon u\cdot v\in \operatorname{Im}(f_i)\},\\
		\widetilde{g}&=(F_i\circ\rho)\left(\d u\d v+\d v\d u\right)+\rho^2\d\Omega^2,
	\end{align*}
	with $F_i(r):=\frac{2K_i}{f'_i(r)}$, $\rho(u,v)=f^{-1}(u\cdot v)$, \cite[Theorem 3.8]{cedwolff}. We note that it is immediate to see that the level-sets of coordinate functions $u$, $v$ are null hypersurface as they carry the degenerate induced metric
	\[
		\rho^2\d\Omega^2.
	\]
	Moreover, the sets $\{u=0\}$, $\{v=0\}$ are contained in the Killing horizon $\{\rho=r_i\}$. For $c>0$, we call the level-sets $\{v=c\}$, $\{u=c\}$ a past-pointing and future-pointing lightcone, respectively. As the constructed spacetime extension contains in fact two copies of $M_i$, $M_{i+1}$ separated either by the axis $\{v=0\}$ or the axis $\{u=0\}$, it suffices to consider the level-sets for $c>0$, as all level sets of $u$ and $v$ for $c<0$ arise by point reflection. In the following, we will only consider a (past-pointing) lightcone $\mathcal{N}=\{v=c\}$, but all identities can be derived for level sets of $u$ in complete analogue.
	
	Following the computations by Roesch \cite[Section 2.2]{roesch} in the case of the Schwarzschild lightcone, we derive all necessary identies for a past-pointing lightcone $\mathcal{N}$. As stated above, for any $c>0$, $\mathcal{N}=\{v=c\}$ is a null hypersurface with degenerate induced metric 
	\[
		\rho^2\d\Omega^2,
	\]
	where $\rho=f^{-1}(u\cdot c)$ along $\mathcal{N}$. Moreover, $\mathcal{N}$ has null generator $\ul{L}=\frac{2K_i}{cF_i}\partial_u$ since \linebreak ${\ul{L}=2K_i\operatorname{grad}(\ln(v))\vert_\mathcal{N}}$. In particular, $\ul{L}(\rho)=1$ and the integral curves of $\ul{L}$ are null geodesics. Hence, $\mathcal{N}$ admits a \emph{background foliation} of round spheres $\Sbb^2_r=\left\{\frac{f_i(r)}{c}\right\}\times\{c\}\times\Sbb^2\subseteq \mathcal{N}$ for $r\in(r_{i-1},r_{i+1})$. Using the explicit identities for the Christoffel symbols as derived in \cite[Proposition B.1]{cedwolff}, it is easy to see that the induced metric $\gamma_r$ and null second fundamental form $\ul{\chi}_r$ (with respect to $\ul{L}$) of the leaves $\Sbb^2_r$ are given by
	\begin{align*}
		\gamma_r&=r^2\d\Omega^2,\\
		\ul{\chi}_r&=r\d\Omega^2.
	\end{align*}
	In particular, $\accentset{\circ}{\ul{\chi}}_r\equiv 0$, and $\mathcal{N}$ is indeed a \emph{shear-free} null hypersurface, cf. \cite{sauter}. Moreover, as it is the case for the Minkowski and Schwarzschild lightcone, see \cite{wolff1} and \cite{roesch} respectively, we note that any spacelike cross section $\Sigma$ of $\mathcal{N}$, i.e., a $2$-dimensional, spacelike submanifold of $\mathcal{N}$ that intersects any integral curve of $\ul{L}$ exactly once, can be uniquely identified with a conformally round metric. To see this, recall that $\ul{L}(\rho)=1$, so $\rho$ restricts to an affine parameter along $\mathcal{N}$, and we may thus uniquely identify $\Sigma$ with a function $\omega\colon \Sbb^2\to(r_{i-1},r_{i+1})$ via
	\[
		\Sigma=\Sigma_\omega=\{(u,v,\vec{x})\in \mathbb{P}^i_h\times\Sbb^2\colon v=c,\rho=f^{-1}(u\cdot c)=\omega(\vec{x})\}.
	\]
	Using \cite[Proposition 1]{marssoria}, we find that the induced metric $\gamma=\gamma_\omega$ and null second fundamental form $\ul{\chi}=\ul{\chi}_\omega$ (with respect to $\ul{L}$) of $\Sigma=\Sigma_\omega$ satisfy
	\begin{align*}
		\gamma_\omega&=\omega^2\d\Omega^2,\\
		\ul{\chi}_\omega&=\omega\d\Omega^2.
	\end{align*}
	In particular, we obtain the following:
	\begin{lem}\label{lem_gausslightcone}
		Let $\Sigma$ be a spacelike cross section of the past-pointing lightcone $\mathcal{N}$. Then
		\[
			\mathcal{H}^2=2\operatorname{R}-2\overline{\operatorname{R}}+4\overline{\Ric}(\ul{L},L)-\overline{\Rm}(\ul{L},L,L,\ul{L}),
		\]
		where $\overline{\Rm}$, $\overline{\Ric}$, $\overline{\operatorname{R}}$ denote the Riemann curvature tensor, Ricci tensor, and scalar curvature of the ambient spacetime, respectively.
	\end{lem}
	\begin{proof}
		As $\ul{\chi}\equiv0$, we note that
		\[
			\newbtr{\vec{\two}}^2=\spann{\ul{\chi},\chi}=\frac{1}{2}\ul{\theta}\theta=\frac{1}{2}\mathcal{H}^2.
		\]
		The claim then directly follows from the Gauss Equation, cf. \cite[Proposition 2.1 (3)]{roesch}.
	\end{proof}
	\begin{bem}\label{bem_gausslightcone}
		Following the computations of Roesch \cite[Section 2.2]{roesch}, we can further compute that
		\begin{align*}
			\chi_r&=rh(r)\d\Omega^2,\\
			\zeta_r&=0,
		\end{align*}
		for the leaves $\Sbb^2_r$ of the background foliaton particular,  and that
		\[
			\mathcal{H}^2=2\operatorname{R}-\frac{4}{\omega^2}\left(1-h(\omega)\right)
		\]
		for any spacelike cross section $\Sigma_\omega$, and the leaf $\Sbb^2_{r_i}$ is a marginally outer trapped surface in $\mathcal{N}$.
	\end{bem}

	Let us now return to the explicit case of the de\,Sitter spacetime, which is of class $\mathcal{S}$ with $h=1-r^2$. As $h(1)=0$ and $h'(1)\not=0$, $\{r=1\}$ is a non-degenerate Killing horizon and we can extend the spacetime past this horizon into the respective generalized Kruskal--Szekeres spacetime, such that the radial function $\rho$ takes any value in $(0,\infty)$. Hence, any past-pointing standard lightcone $\mathcal{N}=\{v=c\}$ ($c>0$) is foliated by a background foliation $(\Sbb^2_r)_{r\in(0,\infty)}$ of round spheres.
	
	Recall that
	\begin{align}\label{eq_riemdesitterlightcone}
	\overline{\Rm}(X,Y,Z,W)=g(X,Z)g(Y,W)-g(Y,Z)g(X,W).
	\end{align}
	Then, using Lemma \ref{lem_gausslightcone} we directly obtain the following, well-known fact:
	\begin{kor}\label{kor_gaussdesitter}
		Let $\Sigma$ be a spacelike cross section of a past-pointing lightcone in the generalized Kruskal--Szekeres extension of the de\,Sitter spacetime. Then
		\[
			\mathcal{H}^2=2\operatorname{R}-4.
		\]
		Thus, for any STCMC surface $\Sigma_\omega$, there exists a constant $b>0$, and vector $\vec{a}\in\R^3$ such that
		\[
			\omega(\vec{x})=\frac{b}{\sqrt{1+\btr{\vec{a}}^2}-\vec{a}\cdot\vec{x}}.
		\]
	\end{kor}
	\begin{proof}
		The identity for $\mathcal{H}^2$ is immediate from \eqref{eq_riemdesitterlightcone} and Lemma \ref{lem_gausslightcone}. In particular, by our previous observations any spacelike cross section $\Sigma$ of $\mathcal{N}$ is an STCMC surface if and only if it is a conformally round surface of constant scalar curvature. As $\omega$ such that $\Sigma=\Sigma_\omega$ is precisely the conformal factor of $\gamma_\omega$, the identity for $\omega$ is a well known fact, see e.g. \cite[Proposition 6]{marssoria}, \cite[Theorem 4.1]{chenwang}.
	\end{proof}
	\begin{bem}\label{bem_gaussdesitter}\,
		\begin{enumerate}
			\item[(i)] Corollary \ref{kor_gaussdesitter} in particular reestablishes the well-known fact that although the value of $\mathcal{H}^2$ is shifted by a constant the STCMC surfaces of the Minkowski and the de\,Sitter lightcone are the same, in the sense that the conformal factor $\omega$ is given by the same identity. Recall that all these surfaces are up to scaling isometric to a standard round sphere via a M\"obius transformation, and that the M\"obius group is isomorphic to the restricted Lorentz group $SO^+(1,3)$. 
			
			On the other hand, the full de\,Sitter spacetime can be realized as the pseudosphere $\mathbb{H}^{1,3}_1(0)$ in the $1+4$-dimensional Minkowski spacetime $\R^{1,4}$. Thus, its full isometry group is $O(1,4)$. However, by our choice of coordinates the subgroup of isometries that leaves the lightcone $\mathcal{N}$ invariant can be identified with $SO^+(1,3)$. So the isometries that leave the lightcone invariant (but transform the individual cross sections) are in $SO^+(1,3)$ both for the de\,Sitter and the Minkowski spacetime. See Section \ref{sec_discussion} for more details.
			\item[(ii)] We recall that $b$ is the area radius of $\Sigma$, i.e., $\btr{\Sigma}=4\pi b^2$, and moreover $\operatorname{R}=\frac{2}{b^2}$. Thus, $\mathcal{H}^2>0$ for $b<1$, $\mathcal{H}^2=0$ for $b=1$, and $\mathcal{H}^2<0$ for $b>1$. Hence, there is a plethora of MOTS in the de\,Sitter lightcone apart from $\Sbb^2_1$, and in fact for any point $p$ on $\mathcal{N}$, we can find a MOTS on $\mathcal{N}$ passing through $p$.
		\end{enumerate}
	\end{bem}
	\newpage
\section{Null mean curvature flow on the de\,Sitter lightcone}\label{sec_nullMCF}
	As first studied by Roesch--Scheuer in \cite{roeschscheuer}, we say a family of spacelike cross sections $\Sigma_t$ is evolving under what we will call \emph{null mean curvature flow} here along a null hypersurface in an ambient spacetime $(M,g)$, if
	\[
		\frac{\d }{\d t}x=\frac{1}{2}g(\vec{\mathcal{H}},L)\ul{L}=-\frac{1}{2}\theta\ul{L}.
	\]
	In this section, we consider the case of a lightcone in the generalized Kruskal--Szekeres extension of the de\,Sitter spacetime as introduced in Section \ref{sec_lightcone}. As $\ul{L}(\rho)=1$ for our choice of null generator, we note that the evolution under null mean curvature flow is equivalent to the scalar parabolic equation
	\begin{align}\label{eq_NMCF}
		\frac{\d }{\d t}\omega=-\frac{1}{2}\theta.
	\end{align}
	In particular, null mean curvature flow along the de\,Sitter lightcone is equivalent to a parabolic evolution equation for the conformal factor $\omega$, exactly as in the case of the standard Minkowski lightcone, cf. \cite{wolff1}. Hence, we can equivalently study the evolution of the (conformally round) metric of the spacelike cross sections. By \cite[Lemma 3.3]{roeschscheuer} and the fact that $\ul{\chi}\equiv 0$, we find 
	\begin{align}\label{eq_metric1}
		\frac{\d }{\d t}\gamma=\ul\theta\left(-\frac{1}{2}\theta\right)\gamma=-\frac{1}{2}\mathcal{H}^2\gamma,
	\end{align}
	which in this case can also be verified by direct computation from \eqref{eq_NMCF} as $\ul{\theta}=\frac{2}{\omega}$. In the Minkowski lightcone, this yields an equivalence between null mean curvature flow and $2d$-Ricci flow of conformally round surfaces, cf. \cite{wolff1}. 
	
	The observation that \eqref{eq_NMCF} and \eqref{eq_metric1} are equivalent remains true for lightcones in general spacetimes of class $\mathcal{S}$ (and their generalized Kruskal--Szekeres extensions) even in higher dimensions. The case, when a MOTS arises as an intersection with a black hole horizon, such as in the Schwarzschild lightcone, is precisely covered by the general analysis of Roesch and Scheuer \cite{roeschscheuer}. In the case of the de\,Sitter lightcone, we have an example of a MOTS that instead arises as an intersection with a cosmological Killing horizon.
	
	By Corollary \ref{kor_gaussdesitter} and \eqref{eq_metric1}, we have that the family of metrics $\gamma(t)$ evolves under null mean curvature flow as
	\begin{align}\label{eq_metric2}
		\frac{\d }{\d t}\gamma=-\left(\operatorname{R}-2\right)\gamma.
	\end{align}
	This gives the following lemma:
	\begin{lem}\label{lem_area}
		Let $(\Sigma_t)$ be a family of spacelike cross sections smoothly evolving under null mean curvature flow starting at $\Sigma_0$. Then
		\[
			\btr{\Sigma_t}=4\pi+e^{2t}\left(\btr{\Sigma_0}-4\pi\right).
		\]
	\end{lem}
	\begin{proof}
		From \eqref{eq_metric2}, we can directly compute that the volume element $\d\mu$ satisfies
		\[
			\frac{\d}{\d t}\d\mu=-(\operatorname{R}-2)\d\mu.
		\]
		Hence,
		\[
			\frac{\d}{\d t}\btr{\Sigma_t}=2\btr{\Sigma_t}-\int_{\Sigma_t}\operatorname{R}=2\btr{\Sigma_t}-8\pi,
		\]
		where we used Gauss--Bonnet. The claim then follows by solving this first order ODE problem.
	\end{proof}
		If we assume $\mathcal{H}^2_t$ is a constant only depending on $t$, one can reduce the parabolic system to an ODE on the area radius. More precisely, for $\vec{a}\in\R^3$, $b_0>0$, we note that
		\[
			\omega_t(\vec{x})=\frac{b(t)}{\sqrt{1+\btr{\vec{a}}^2}-\vec{a}\cdot\vec{x}}
		\]
		is a solution to null mean curvature flow, if and only if
		\[
			b(t)=\sqrt{1+e^{2t}( b_0^2-1)}.
		\]
		It is easy to see that these families of (round) spheres are shrinking if $b_0<1$, stationary if $b_0=1$, and expanding if $b_0>1$. Let $\omega_{b_0,\vec{a}}$ denote the conformal factor of an STCMC surface with respect to $b_0>0$, $\vec{a}\in\R^3$ as given by the formula in Corollary \ref{kor_gaussdesitter}. By the avoidance principle we find for any solution to null mean curvature flow that if $\omega_0\le \omega_{b_0,\vec{a}}$ for some $b_0<1$ the solution is shrunk towards the cone as long as it remains smooth, and if $\omega_0\ge \omega_{b_0,\vec{a}}$ for some $b_0>1$ the solution expands towards infinity as long as it remains smooth.
\subsection{Singularity formation and asymptotic behavior}\label{subsec_singularities}
	We now want to gain a precise understanding of the formation of singularities and the asymptotic behaviour of solution to null mean curvature flow. The key ingredients to our analysis will be the precise formula for the area growth, Lemma \ref{lem_gausslightcone}, and a rescaling to volume preserving $2d$-Ricci flow.
	
	To this end, for any smooth solution $(\Sigma_t)$ of null mean curvature flow, we define the scaling factor
	\[
		c(t):=\exp\left(\left(\int_0^t\fint_{\Sigma_s}\operatorname{R}\d s\right)-2t\right)>0,
	\]
	and rescaled time $\widetilde{t}(t):=\int_0^tc(s)\d s$, where $\fint_\Sigma \operatorname{R}=\frac{\int_\Sigma \operatorname{R}}{\btr{\Sigma}}$. Note that $\widetilde{t}$ is invertible with smooth inverse as $c$ is a strictly positive, smooth function. We then define the family of metrics $(\widetilde{\gamma}(\,\widetilde{t}\,))$ as
	\[
		\widetilde{\gamma}(\,\widetilde{t}\,)=c(t)\gamma(t).
	\]
	Then,
	\begin{align*}
		\frac{\d}{\d\widetilde{t}}\widetilde{\gamma}(\,\widetilde{t}\,)&=\frac{1}{c(t)}\frac{d}{\d t}\left(c(t)\gamma(t)\right)=\left(\fint \operatorname{R}-2\right)\gamma(t)-\left(\operatorname{R}-2\right)\gamma(t)=-\left(\widetilde{\operatorname{R}}-\fint\widetilde{\operatorname{R}}\right)\widetilde{\gamma}(\,\widetilde{t}\,),
	\end{align*}
	where we used the scaling properties of the scalar curvature in the last line. Hence, the family of metrics $(\widetilde{\gamma}(\,\widetilde{t}\,))$ evolves under volume preserving Ricci flow. By a classical result first proven by Hamilton \cite{hamilton1}, the solution of volume preserving Ricci flow starting from any metric on a compact Riemannian manifold exists for all times and converges to a metric of constant (scalar) curvature as $\widetilde{t}\to\infty$. In the conformally round case, this was initially proven by Hamilton under the assumption of strictly positive scalar curvature, which was removed by Chow \cite{chow1}. The conformally round case was later revisited and independent proofs were also given by Bartz--Struwe--Ye \cite{bartzstruweye}, Struwe \cite{struwe}, Andrews-Bryan \cite{andrewsbryan}, and recently by the author \cite{wolff1} under Hamiltons initial assumption of positive scalar curvature.
	Note that using Lemma \ref{lem_gausslightcone} we can explicitly compute that
	\begin{align}\label{eq_areagrowth}
		c(t)=\frac{\btr{\Sigma_0}}{\btr{\Sigma_t}}=\frac{\btr{\Sigma_0}}{4\pi+e^{2t}(\btr{\Sigma_0}-4\pi)},
	\end{align}
	which is consistent with the fact that the rescaled flow is preserving the area by construction. This precise formula for the scaling factor now allows for a complete analysis of the singularity formation and asymptotic behaviour of null mean curvature flow on the de\,Sitter lightcone.
	
	First, we prove a preliminary result.
	\begin{prop}\label{prop_nosing}
		Let $\Sigma_0$ be a spacelike cross section of the de\,Sitter lightcone $\mathcal{N}$. Then, the solution of null mean curvature flow along $\mathcal{N}$ starting at $\Sigma_0$ develops no singularities in finite time $T_{max}<\infty$ unless $\btr{\Sigma_t}\to 0$ as $t\to T_{max}$.
	\end{prop}
	\begin{proof}
		Assume that the smooth solution $(\Sigma_t)$ starting at $\Sigma_0$ develops a singularity in finite time such that $\btr{\Sigma_t}\not\to0$ as $t\to T_{max}<\infty$. In particular, $c(t)$ and hence $\widetilde{t}(t)$ remain bounded as $t\to T_{max}$ which implies that volume preserving $2d$-Ricci flow starting at $\Sigma_0$ develops a finite time singularity. This gives an immediate contradiction, as the solution exists for all times without forming a singularity, cf. \cite[Corollary 1.3]{chow1}.
	\end{proof}
	In particular, a solution of null mean curvature flow exists for all positive times, unless $\btr{\Sigma_0}<4\pi$, in which case the solutions forms a singularity in finite time with $T_{max}$ fully determined by the explicit formula for $\btr{\Sigma_t}$, Lemma \ref{lem_gausslightcone}. We now state our main theorem.
	\begin{thm}\label{thm_main1}
		Let $\Sigma_0$ be a spacelike cross section of the lightcone $\mathcal{N}$ in the de\,Sitter spacetime, and let $(\Sigma_t)$ be the solution of null mean curvature flow along $\mathcal{N}$ starting at $\Sigma_0$. Then,
		\begin{enumerate}
			\item[\emph{(i)}] if $\btr{\Sigma_0}<4\pi$ the solution shrinks into the tip of the cone in finite time. Moreover, the rescaled solution converges to a surface of constant spacetime mean curvature with $\mathcal{H}^2>0$ as $\widetilde{t}\to\infty$.
			\item[\emph{(ii)}] if $\btr{\Sigma_0}=4\pi$ the flow exists for all times and converges to a MOTS as $t\to\infty$.
			\item[\emph{(iii)}] if $\btr{\Sigma_0}>4\pi$ the flow exists for all times and expands towards infinity. The rescaled flow converges to a conformally round surface, but unless $(\Sigma_t)$ is a family of expanding spheres, the limit is not a surface of constant spacetime mean curvature.
		\end{enumerate}
	\end{thm}
	
	\begin{proof}
		\begin{enumerate}
			\item[(i)] Let $\btr{\Sigma}<4\pi$. As stated in Lemma \ref{lem_gausslightcone}, we have
			\[
				\btr{\Sigma_t}=4\pi+e^{2t}\left(\btr{\Sigma_0}-4\pi\right).
			\]
			Thus $\btr{\Sigma_t}\to0$ as $t\to T_{max}:= \frac{1}{2}\ln\left(\frac{4\pi}{4\pi-\btr{\Sigma_0}}\right)$. By Proposition \ref{prop_nosing} we know that no singularity can arise for times $t<T_{max}$. Recall that the time scale $\widetilde{t}(t)$ for the flow rescaled to be volume preserving is given by 
			\[
				\widetilde{t}(t)=\int\limits_0^tc(s)\d s,
			\]
			where $c(t)$ is given by \eqref{eq_areagrowth}. Integration yields that
			\[
				\widetilde{t}(t)=\frac{\btr{\Sigma_0}}{4\pi}\left(t+\frac{1}{2}\ln\left(\frac{\btr{\Sigma_0}}{4\pi+e^{2t}\left(\btr{\Sigma_0}-4\pi\right)}\right)\right).
			\]
			It is straightforward to check that $\widetilde{t}(t)\to\infty$ as $t\to T_{max}$. In particular, the rescaled flow exists for all times and thus converges to a surface of constant scalar curvature, cf. \cite{chow1}. By Corollary \ref{kor_gaussdesitter}, the rescaled flow converges to an STCMC surface.
			\item[(ii)] Let $\btr{\Sigma_0}=4\pi$. By Lemma \ref{lem_gausslightcone}, we have
			\[
				\btr{\Sigma_t}=4\pi,
			\]
			so the area is preserved, and it is indeed straightforward to observe that the flow is equivalent to volume preserving Ricci flow starting from $\Sigma_0$. Thus, the solution exists for all times and converges to a surface of constant scalar curvature, cf. \cite{chow1}. As the area is preserved and $\btr{\Sigma_0}=4\pi$, we know that the limiting scalar curvature satisfies $\operatorname{R}_\infty=2$. Hence, the limiting surface is a MOTS by Corollary \ref{kor_gaussdesitter}.
			\item[(iii)] Let $\btr{\Sigma_0}>4\pi$. In particular,
			\[
				\btr{\Sigma_t}=4\pi+e^{2t}\left(\btr{\Sigma_0}-4\pi\right)\ge 4\pi
			\]
			by Lemma \ref{lem_gausslightcone}. Thus, the flow exits for all positive times by Proposition \ref{prop_nosing} and expands towards infinity, i.e., $\btr{\Sigma_t}\to\infty$ as $t\to\infty$.
			
			Now, as in case (i), we note that the rescaled time $\widetilde{t}$ satisfies
			\[
				\widetilde{t}(t)=\frac{\btr{\Sigma_0}}{4\pi}\left(t+\frac{1}{2}\ln\left(\frac{\btr{\Sigma_0}}{4\pi+e^{2t}\left(\btr{\Sigma_0}-4\pi\right)}\right)\right).
			\]
			In particular, as $\btr{\Sigma_0}>4\pi$, we have
			\[
				e^{2\widetilde{t}(t)}=e^{\frac{\btr{\Sigma_0}}{4\pi}}\frac{e^{2t}}{4\pi+e^{2t}(\btr{\Sigma_0}-4\pi)}\to e^{\frac{\btr{\Sigma_0}}{4\pi}}\frac{1}{\btr{\Sigma_0}-4\pi}<\infty
			\]
			as $t\to\infty$. Hence $\widetilde{t}\to\widetilde{t}(\btr{\Sigma_0})$ as $t\to\infty$ for some constant $\widetilde{t}(\btr{\Sigma_0})>0$ only depending on $\btr{\Sigma_0}$. Thus, the rescaled flow converges to a smooth conformally round surface, namely $\widetilde{\Sigma}_{\widetilde{t}({\btr{\Sigma_0}})}$, but is not round unless the rescaled flow converges to a round limit in finite time. By the strong maximum principle applied to the evolution of $\left(\mathcal{H}^2-\fint\mathcal{H}^2\right)$  under the rescaled equation, this only occurs when the family $\widetilde{\Sigma}_t$ is a family of stationary round spheres, i.e., a family of expanding spheres under the unrescaled flow.
		\end{enumerate}
	\end{proof}
\subsection{Ancient solutions}\label{subsec_ancientSol}
	We are further able to completely characterize the ancient solutions to null mean curvature flow along the de\,Sitter lightcone. As before, this will directly follow from a suitable rescaling. Similar to the above, one can check that if $(\Sigma_t)$ is a solution of null mean curvature flow, then the rescaled metrics
	\[
		\widehat{\gamma}(\,\widehat{t}\,)=C(t)\gamma(t),
	\]
	where
	\begin{align}
		C(t)&:=e^{-2t},\label{eq_scaling_ricci}\\
		\widehat{t}(t)&:=\int\limits_0^tC(s)\d s=\frac{1}{2}\left(1-e^{-2t}\right),\label{eq_time_ricci}
	\end{align}
	are evolving under $2d$-Ricci flow (in the conformal class of the round sphere). Now, if $(\Sigma_t)$ is an ancient solution, i.e., the solution exists for all negative times, then the family $\left(\widehat{\gamma}(\,\widehat{t}\,)\right)$ is an ancient solution to $2d$-Ricci flow, as $\widehat{t}(t)\to-\infty$ for $t\to-\infty$ by \eqref{eq_time_ricci}.
	
	By a maximum principle argument, one can see that all compact ancient solutions to $2d$-Ricci flow are either flat, in which case they are eternal solutions, or the scalar curvature is strictly positive for all times, and hence the metrics are conformally round. In the latter case, Daskalopoulos--Hamilton--Sesum \cite{daskahamilsesum} characterized all ancient solutions with strictly positive scalar curvature as either a family of shrinking spheres or the King--Rosenau solution. Up to a M\"obius transformation, one can identify them with the following families of conformal factors
	\begin{align}
		\widehat{\omega}_{Sph}(\,\widehat{t}\,,\theta,\varphi)&=\sqrt{-2\widehat{t}\,\,}\label{eq_shrinkingspheres},\\
		\widehat{\omega}_{K-R}(\,\widehat{t}\,,\theta,\varphi)&=\sqrt{\frac{2\sinh(-\widehat{t}\,)}{1+\frac{\sin^2\theta}{2}\left(\cosh(-\widehat{t}\,)-1\right)}}\,\label{eq_kingrosenau},
	\end{align}
	defined on $(-\infty,0)$. Note that in both cases, we have that the area $\btr{\widehat{\Sigma}_t}$ of the evolving surfaces satisfies $\btr{\widehat{\Sigma}_t}\to-\infty$ as $\widehat{t}\to\infty$, and $\btr{\widehat{\Sigma}_t}\to0$ as $\widehat{t}\to0$. Of course, for any ancient solution, a translation in time $\widehat{t}\mapsto \widehat{t}+a$ ($a\in\R$) remains an ancient solution, as this merely corresponds to a relabelling of the family of metrics. Thus, depending on $a$, we can choose a family of shrinking spheres or a King--Rosenau solution such that at $t=0$ we have that $\btr{\widehat{\Sigma}_0}<4\pi$, $\btr{\widehat{\Sigma}_0}=4\pi$, or $\btr{\widehat{\Sigma}_0}>4\pi$. We obtain the following characterization using Theorem \ref{thm_main1}.
	\begin{thm}\label{thm_main2}
		Up to Lorentz transformations, there are six ancient solutions of null mean curvature flow in the de\,Sitter lightcone $\mathcal{N}$, three of which are rescalings of a familiy of shrinking spheres and three of which are rescalings of a King--Rosenau solution. Four of these solutions are eternal, one of which does not become round as $t\to\infty$.
		
		Up to the aforementioned Lorentz transformation, all three families of shrinking spheres converge to the MOTS $\Sigma_1$ as $t\to-\infty$, and the three types of King--Rosenau solutions have the same asymptotic behavior as $t\to-\infty$ with area converging to $4\pi$.
	\end{thm}
	\begin{bem}\label{bem_main2}
		From Lemma \ref{lem_area} is is immediate that any ancient solutions satisfies $\btr{\Sigma_t}\to4\pi$ as $t\to-\infty$. Moreover, it is direct to see that the different cases in Theorem \ref{thm_main1} are naturally preserved under the flow, i.e., $\btr{\Sigma_t}<4\pi$ ($=4\pi$, $>4\pi$) if and only if $\btr{\Sigma_0}<4\pi$ ($=4\pi$, $>4\pi$). So even though two families of spheres or King--Rosenau solutions are related directly by a relabelling of time $\widehat{t}\mapsto \widehat{t}+a$, the rescaled surfaces evolving under null mean curvature flow will in general have two distinct behaviors.
	\end{bem}
	\begin{proof}
		Let $(\Sigma_t)$ be an ancient solution of null mean curvature flow along the de\,Sitter lightcone, and wlog we assume that the flow exists at time $0$. As described above, we can rescale $(\Sigma_t)$ such that the family of rescaled metrics $\left(\widehat{\gamma}(\,\widehat{t}\,)\right)$ is an ancient solution of $2d$-Ricci flow in the conformal class of the round sphere. Hence, by the work of Daskalopoulos--Hamiltn--Sesum \cite{daskahamilsesum} it is either a family of shrinking spheres or a King--Rosenau solution up to a suitable Lorentz transformation, which precisely acts as a M\"obius transformation on the conformally round surfaces, cf. Remark \ref{bem_gaussdesitter} (i).
		
		In particular, by Theorem \ref{thm_main1} the two ancient solutions characterized by Daskalopoulos--Hamilton--Sesum further split into three cases upon scaling back to null mean curvature flow depending on whether $\btr{\Sigma_0}<4\pi$, $\btr{\Sigma_0}=4\pi$, or $\btr{\Sigma_0}>4\pi$. Hence, there is a total of six ancient solutions. Further, Theorem \ref{thm_main1} immediately yields that four of them are eternal (in the cases when $\btr{\Sigma_0}\ge 4\pi$), and that an ancient solution with $\btr{\Sigma_0}>4\pi$, and such that it can be rescaled to a King--Rosenau solution, can not become round as $t\to\infty$. Note that as $C(0)=1$ and $\widehat{t}(0)=0$, we have $\widehat{\gamma}(0)=\gamma_0$. So conversely, all six cases indeed occur by the discussion of ancient solutions to $2d$-Ricci flow above, and scaling back to null mean curvature flow.
		
		It remains to discuss the asymptotic behaviour as $t\to-\infty$. If we denote the conformal factors of the solution by $\omega$ and the conformal factors of the rescaled surfaces evolving under $2d$-Ricci flow by $\widehat{\omega}$, then \eqref{eq_scaling_ricci} and \eqref{eq_time_ricci} yield the precise relation
		\[
		\omega(t,\cdot)=C(t)^{-1}\widehat{\omega}\left(\,\widehat{t}\,\cdot\right)=e^t\widehat{\omega}\left(\frac{1}{2}\left(1-e^{2t}\right),\cdot\right).
		\]
		Therefore, as discussed above, there exists a suitable Lorentz transformation and $\widehat{t}_0>0$ such that either
		\[
			\omega(t,\cdot)= e^t\widehat{\omega}_{Sph}(\left(\frac{1}{2}\left(1-e^{2t}\right)-\widehat{t}_0,\cdot\right)\text{\,\,\,\,\,\,\,or\,\,\,\,\,\,\, }\omega(t,\cdot)=e^t\widehat{\omega}_{K-R}(\left(\frac{1}{2}\left(1-e^{2t}\right)-\widehat{t}_0,\cdot\right).
		\]
		Using the explicit formulas $\eqref{eq_shrinkingspheres}$ and $\eqref{eq_kingrosenau}$ one can check directly that the limits as $t\to-\infty$ are independent of $t_0$, so all three families of shrinking spheres, and the three types King--Rosenau solutions have the same asmptotic behaviour in either case. In particular, in the case of spheres it is immediate to see that $\omega(t,\cdot)\to 1$ as $t\to-\infty$, so the spheres converge to the MOTS $\Sigma_1$ as claimed. We note that the limit becomes singular for a King--Rosenau solution, so the ancient solutions do not converge to a regular limiting surface in this case. Nonetheless, the area converges to $4\pi$ by Remark \ref{bem_main2}.
	\end{proof}
	
\section{Comments}\label{sec_discussion}

	For the purpose of this paper, we have considered the de\,Sitter spacetime $(\mathfrak{M},\mathfrak{g})$ as a spacetime of class $\mathcal{S}$, i.e., of the form
	\begin{align*}
		M&=\R\times(0,1)\times\Sbb^2,\\
		g&=-(1-r^2)\d t^2+\frac{1}{1-r^2}\d r^2+r^2\d\Omega^2,
	\end{align*}
	which we extended past the Killing horizon $\{r=1\}$ into a generalized Kruskal--Szekeres extension, see \cite{cedwolff}, to study the round lightcones on the full range $r\in(0,\infty)$.
	
	On the other hand, the de\,Sitter spacetime can be realized as the \emph{one-sheeted hyperboloid} or \emph{pseudosphere} $\mathbb{H}^{1,3}_1$ of radius one in the $1+4$-dimensional Minkowski spacetime $\R^{1,4}$, i.e.,
	\[
		\mathbb{H}^{1,3}_1=\{p\in \R^{1,4}\colon \eta(p,p)=1\},
	\]
	where we imbue $\mathbb{H}^{1,3}_1$ with a metric of constant sectional curvature $1$ via the induced metric from $\R^{1,4}$.
	From this, it is direct to recall the well-known fact that the full isometry group of the de\,Sitter spacetime is $O(1,4)$. Now let $X\in \mathbb{H}^{1,3}_1$, and after a suitable Lorentz transformation we may assume that $X=x_1=(0,1,0,0,0)$. Then, upon a choice of Cartesian coordinates $x_0$, $x_1$, $x_2$, $x_3$, $x_5$ with $x_1=X$, it is a well-known fact that two copies of $(\mathfrak{M},\mathfrak{g})$ can indeed be isometrically embedded into $\mathbb{H}^{1,3}_1$ via the parametrizations
	\[
		f_{1,\pm}\colon \R\times(0,1)\times\Sbb^2\to \mathbb{H}^{1,3}_1\colon (t,r,\vec{x})\mapsto \left(\pm\sqrt{1-r^2}\sinh(t),\pm\sqrt{1-r^2}\cosh(t),r\cdot\vec{x}\right),
	\]
	which yields static coordinates on $\operatorname{Im}(f_{1,\pm})\subsetneq \mathbb{H}^{1,3}_1$. Note that $f_{1,\pm}(t,r,\vec{x})\to (0,0,\vec{x})$ as $r\to1$. Now there are exactly two future-pointing null geodesic lines starting at $(0,0,\vec{x})$ with speed $\partial_0\pm\partial_1$ contained in $\mathbb{H}^{1,3}_1$, and the collection of all such null geodesics starting from $\{0\}\times\{0\}\times\Sbb^2$ form two smooth null hypersurfaces which intersect at $\{0\}\times\{0\}\times\Sbb^2$ and whose union (together with $\{r=0\}$) is the boundary of $\operatorname{Im}(f_{1,-})\cup\operatorname{Im}(f_{1,+})$ in $\mathbb{H}^{1,3}_1$. Thus, the two copies of $(\mathfrak{M},\mathfrak{g})$ are bounded by two intersecting null hypersurfaces exactly as in the generalized Kruskal--Szekeres extension. Indeed, considering the parametrizations
	\[
		f_{2,\pm}\colon \R\times(1,\infty)\times\Sbb^2\to \mathbb{H}^{1,3}_1\colon (t,r,\vec{x})\mapsto \left(\pm\sqrt{r^2-1}\cosh(t),\pm\sqrt{r^2-1}\sinh(t),r\cdot\vec{x}\right)
	\]
	which embed isometrically into $\mathbb{H}^{1,3}_1$ (equipped with the same metric $g$), and which are also bounded by the same two intersecting null hypersurfaces. We note that these null hypersurfaces correspond to the Killing horizon in the generalized Kruskal--Szekeres extension, and that their intersection $\{0\}\times\{0\}\times\Sbb^2$ corresponds to the bifurcation surface. Hence, the generalized Kruskal--Szekeres extension precisely recovers the full de\,Sitter spacetime in the sense that it covers all of $\mathbb{H}^{1,3}_1$ (except for two meridians corresponding to $\{r=0\}$). In particular, along the Killing horizon we find $\partial_u =\partial_0+\partial_1$, $\partial_v=\partial_0-\partial_1$. See Figure \ref{figure} below.
	
	\begin{figure}[H]
		\centering
		\includegraphics[scale=0.7]{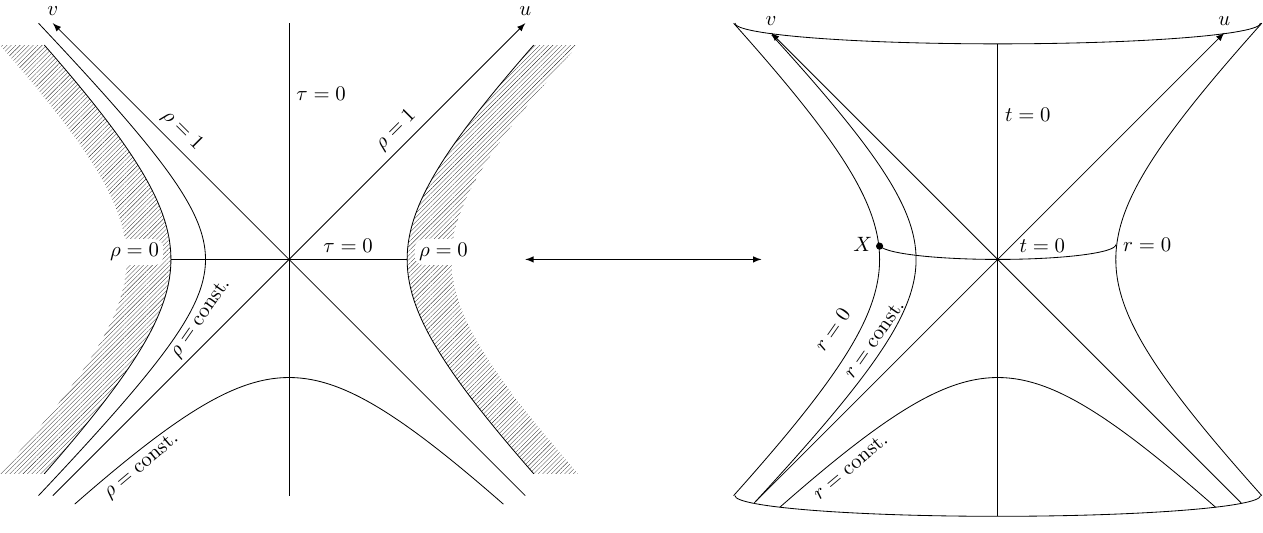}
		\caption{The generalized Kruskal--Szerekes extension of the de\,Sitter spacetime on the left, and the pseudosphere $\mathbb{H}_1^{1,3}$ on the right (supressing two spatial dimensions), with corresponding $\tau$, $\rho$ and $t$, $r$ coordinate lines.}
		\label{figure}
	\end{figure}
	
	Note moreover, that one can realize $\mathbb{H}^{1,3}_1$ as the quotient $\frac{O(1,4)}{O(1,3)}$ in the following way: For $X$ as above, consider the surjective map
	\[
		\Pi\colon =O(1,4)\to \mathbb{H}^{1,3}_1\colon L\mapsto LX.
	\]
	Now, we identify the subgroup of $O(1,4)$ that acts solely on the coordinates $x_0$, $x_2$, $x_3$, $x_4$ with $O(1,3)$. Then, we have $\Pi(L)=X$ if and only if $L\in O(1,3)$ by definition. As $\Pi(L_1\circ L_2)=\Pi(L_1)(\Pi(L_2))$, we have $\Pi(L\circ\widetilde{L})=\Pi(L)$ for all $L\in O(1,4)$, $\widetilde{L}\in O(1,3)$, so the map $\Pi(LO(1,3)):=\Pi(L)$ is well-defined and bijective. Note that both this bijective map and our choice of generalized Kruskal--Szekeres coordinates on $\mathbb{H}^{1,3}_1$ depend on our choice of $X$. In particular, as suggested by Remark \ref{bem_gaussdesitter} (i), the subgroup of the full isometry group of $\mathbb{H}^{1,3}_1$ that leaves the de\,Sitter lightcone under consideration , i.e., a principal null hypersurface in the generalized Kruskal--Szekeres spacetime, (and its orientation) invariant is the restricted Lorentz group $\operatorname{SO}^+(1,3)\subseteq \operatorname{O}(1,3)\subseteq \operatorname{O}(1,4)$, precisely as it is the case for the Minkowski lightcone. However, although $L\in \operatorname{SO}^+(1,3)$ leaves both $X$ and the lightcone invariant, it will in general not leave $\partial_0$ and therefore the Killing horizon (and the Kruskal--Szekeres extension as a whole) invariant, which leads to the plethora of MOTS observed in the de\,Sitter lightcone.
	
	This relation between the Kruskal--Szekeres extension and the pseudosphere is retained in higher dimensions. Moreover, similar to the higher dimensional Minkowski case, see \cite[Section 6]{wolff1}, null mean curvature flow along the de\,Sitter lightcone can be rescaled to the Yamabe flow in the conformal class of the round sphere $S^{n-1}$ for all $n\ge 3$. In particular, one can appeal to a short proof in the conformally round case by Brendle \cite{brendle2}.
	
	We note moreover that as $\mathbb{H}^{1,3}_1$ lies totally umbilic in $\R^{1,4}$ with umbilicity factor $1$ (with respect to a spacelike unit normal $\vec{n}$), we have that the codimension-$3$ mean curvature vector $\vec{\mathcal{H}}_3$ of any spacelike cross section of the de\,Sitter lightcone as a codimension-$3$ surface in the ambient $\R^{1,4}$ is given by
	\[
		\vec{\mathcal{H}}_3=\vec{\mathcal{H}}-\vec{n},
	\]
	and therefore 
	\[
		\eta\left(\vec{\mathcal{H}}_3,\vec{\mathcal{H}}_3\right)=\mathcal{H}^2+4=2\operatorname{R},
	\]
	cf. Corollary \ref{kor_gaussdesitter}. In this sense, we recover the Gauss equation in the Minkowski lightcone, see \cite[Equation 5]{wolff1}, also for spacelike cross sections of the de\,Sitter lightcone by considering them not as codimension-$2$ surfaces in the de\,Sitter spacetime but as codimension-$3$ surfaces in $\R^{1,4}$. Again, this relation remains true in higher dimensions $n\ge 3$ up to constants depending on $n$.
	
	Lastly, let us comment on the case of the Anti de\,Sitter spacetime. Again, we can both realize the Anti de\,Sitter spacetime as a spacetime of class $\mathcal{S}$ corresponding to 
	\[
	h\colon(0,\infty)\to\R \colon r\mapsto 1+r^2,
	\]
	and as the pseudosphere $\mathbb{H}^{2,2}_1$ in the semi-Riemannian flat manifold $\R^{2,3}$ with index $2$. Note that $h$ has no zeroes and indeed the static coordinates cover all of $\mathbb{H}^{2,2}_1$ except for a set of measure zero. Nontheless, we can consider null mean curvature flow along a round lightcone in the Anti de\,Sitter spacetime as before, and for any (conformally round) spacelike cross section the Gauss Equation then yields
	\[
		\mathcal{H}^2=2\operatorname{R}+4.
	\]
	In particular, we can realize null mean curvature flow again as a rescaling of $2d$-Ricci flow, although due to the change of sign in the constant the flow will shrink towards the tip of the cone even faster. More precisely, we have 
	\[
		\btr{\Sigma_t}=(4\pi+\btr{\Sigma_0})e^{-2t}-4\pi
	\]
	for solutions, so any solution must develop singularities in finite time. Arguing by rescaling as in Section \ref{sec_nullMCF}, one can check that qualitatively the behaviour of null mean curvature flow in the Anti de\, Sitter lightcone regarding singularity formation and ancient solutions is the same as for null mean curvature flow in the Minkowski lightcone, that is to say $2d$-Ricci flow.

\bibliography{bib_desitterlightcone}

\begin{thebibliography}{10}

\bibitem{andrewsbryan}
Ben Andrews and Paul Bryan.
\newblock {Curvature bounds by isoperimetric comparison for normalized Ricci
  flow on the two-sphere}.
\newblock {\em Calculus of Variations and Partial Differential Equations},
  39:419--428, 2009.

\bibitem{bartzstruweye}
Janet Bartz, Michael Struwe, and Runguo Ye.
\newblock {A new approach to the Ricci flow on $\mathbb{S}^2$}.
\newblock {\em Annali Della Scuola Normale Superiore Di Pisa-classe Di
  Scienze}, 21:475--482, 1994.

\bibitem{birmi}
Danny Birmingham.
\newblock {Topological black holes in anti-de Sitter space}.
\newblock {\em Class. Quantum Grav.}, 16(4):1197, 1999.

\bibitem{brendle2}
Simon Brendle.
\newblock {A short proof of the Yamabe flow on $S^n$}.
\newblock {\em Pure and Applied Mathematics Quarterly}, 3(2):499--512, 2007.

\bibitem{brillhayward}
Dieter~R. Brill and Sean~A. Hayward.
\newblock Global structure of a black hole cosmos and its extremes.
\newblock {\em Class. Quantum Grav.}, 11:359--370, 1993.

\bibitem{cedgal}
Carla Cederbaum and Gregory~J. Galloway.
\newblock Photon surfaces with equipotential time slices.
\newblock {\em J. Math. Phys.}, 62, 2021.

\bibitem{cederbaumsakovich}
Carla Cederbaum and Anna Sakovich.
\newblock {On the center of mass and foliations by constant spacetime mean
  curvature surfaces for isolated systems in General Relativity}.
\newblock {\em Calculus of Variations and Partial Differential Equations},
  60(214), 2021.

\bibitem{cedwolff}
Carla Cederbaum and Markus Wolff.
\newblock Some new perspectives on the kruskal--szekeres extension.
\newblock {\em arXiv:2310.06946}, 2023.

\bibitem{chenwang}
Po-Ning Chen and Ye-Kai Wang.
\newblock {Two rigidity results for surfaces in Schwarzschild spacetimes}.
\newblock {\em arXiv:2306.07477}, 2023.

\bibitem{chow1}
Bennet Chow.
\newblock {The Ricci flow on the $2$-sphere}.
\newblock {\em Journal of Differential Geometry}, 33(2):325--334, 1991.

\bibitem{chrugalpot}
Piotr~T. Chru\'{s}ciel, Gregory~J. Galloway, and Yohan Potaux.
\newblock {Uniqueness and energy bounds for static AdS metrics}.
\newblock {\em Phys. Rev. D}, 101:064034, 2020.

\bibitem{daskahamilsesum}
Panagiota Daskalopoulos, Richard Hamilton, and Natasa Sesum.
\newblock {Classification of ancient compact solutions to the Ricci flow on
  surfaces}.
\newblock {\em J. Diff. Geom.}, 91(2):171--214, 2012.

\bibitem{gibbhawk}
Gary~W. Gibbons and Stephen~W. Hawking.
\newblock {Cosmological event horizons, thermodynamics, and particle creation}.
\newblock {\em Phys. Rev. D.}, 15(10), 1977.

\bibitem{hamilton1}
Richard~S. Hamilton.
\newblock {The Ricci flow on surfaces}.
\newblock {\em Math and General Relativity, Contemporary Mathematics},
  71:237--262, 1988.

\bibitem{kottler}
Friedrich Kottler.
\newblock {\"{U}ber die physikalischen Grundlagen der Einsteinschen
  Gravitationstheorie}.
\newblock {\em Annalen der Physik}, 56:401--462, 1918.

\bibitem{marssoria}
Marc Mars and Alberto Soria.
\newblock {The asymptotic behavior of the Hawking energy along null
  asymptotically flat hypersurfaces}.
\newblock {\em Classical and Quantum Gravity}, 32(18):185020, 2015.

\bibitem{roesch}
Henri Roesch.
\newblock {Proof of a null Penrose conjecture using a new quasi-local mass}.
\newblock {\em Communications in Analysis and Geometry}, 29(8):1847--1915,
  2021.

\bibitem{roeschscheuer}
Henri Roesch and Julian Scheuer.
\newblock {Mean Curvature Flow in Null Hypersurfaces and the Detection of
  MOTS}.
\newblock {\em Communications in Mathematical Physics}, 390:1--25, 2022.

\bibitem{sauter}
Johannes Sauter.
\newblock {Foliations of null hypersurfaces and the Penrose inequality}.
\newblock {\em Doctoral thesis, ETH Z\"urch}, 2008.

\bibitem{schindagui}
J.~C. Schindler and A.~Aguirre.
\newblock Algorithms for the explicit computation of penrose diagrams.
\newblock {\em Class. Quantum Grav.}, 35, 2018.

\bibitem{struwe}
Michael Struwe.
\newblock {Curvature flows on surfaces}.
\newblock {\em Annali Della Scuola Normale Superiore di Pisa-classe Di
  Scienze}, Ser. 5, 1(2):247--274, 2002.

\bibitem{wald}
Robert~M. Wald.
\newblock {\em {General Relativity}}.
\newblock The University of Chicago Press, Chicago 60637, 1984.

\bibitem{wangwangzang}
M.~Wang, Y.~Wang, and Y.~Zhang.
\newblock Minkowski formulae and alexandrov theorems in spacetime.
\newblock {\em J. Diff. Geom.}, 105:249--290, 2017.

\bibitem{wolff1}
Markus Wolff.
\newblock {Ricci flow on surfaces along the standard lightcone in the
  $3+1$-Minkowski spacetime}.
\newblock {\em Calculus of Variations and Partial Differential Equations},
  62(90), 2023.

\end{thebibliography}

\nopagebreak
\bibliographystyle{plain}
\end{document}